\documentclass[10pt,a4paper,oneside,reqno]{amsart}

\usepackage{setspace}
\usepackage[totalwidth=14cm,totalheight=22cm]{geometry}
\usepackage[T1]{fontenc}
\usepackage[dvips]{graphicx}
\usepackage{epsfig}
\usepackage{amsmath,amssymb,amscd,amsthm}
\usepackage{subfigure}
\usepackage[latin1]{inputenc}
\usepackage{latexsym}
\usepackage{graphics}
\usepackage{colortbl} 
\usepackage{color}
\usepackage{ae}
\usepackage{enumitem}
\usepackage[all]{xy}
\usepackage{scalerel}
\usepackage{tikz}
\usepackage{cancel}
\usepackage{bbm,amsbsy}
\usepackage{mathrsfs}  
\usepackage[colorlinks=true,pagebackref=true]{hyperref}
\usepackage[normalem]{ulem}
\usepackage{nicefrac}
\usepackage{xfrac}
\usepackage{faktor}
\usepackage{tikz-cd} 

\usepackage{nomencl}
\usepackage[normalem]{ulem}

\makenomenclature

\usepackage{bbm}

\makeatletter
\newtheorem*{rep@theorem}{\rep@title}
\newcommand{\newreptheorem}[2]{%
\newenvironment{rep#1}[1]{%
 \def\rep@title{#2 \ref{##1}}%
 \begin{rep@theorem}}%
 {\end{rep@theorem}}}
\makeatother

\makeatletter
\newtheorem*{rep@cor}{\rep@title}
\newcommand{\newrepcor}[2]{%
\newenvironment{rep#1}[1]{%
 \def\rep@title{#2 \ref{##1}}%
 \begin{rep@cor}}%
 {\end{rep@cor}}}
\makeatother

\makeatletter
\newtheorem*{rep@prop}{\rep@title}
\newcommand{\newrepprop}[2]{%
\newenvironment{rep#1}[1]{%
 \def\rep@title{#2 \ref{##1}}%
 \begin{rep@prop}}%
 {\end{rep@prop}}}
\makeatother

\newlist{steps}{enumerate}{1}
\setlist[steps, 1]{itemsep=8pt,leftmargin=0cm,itemindent=.5cm,labelwidth=\itemindent,labelsep=0cm,align=left,label = \textbf{\emph{Step \arabic*}:\,}}

\makeatletter
\newcommand{\myitem}[1]{%
\item[#1]\protected@edef\@currentlabel{#1}%
}
\makeatother

\newcommand{\II}{I\hspace{-0.1cm}I}
\newcommand{\III}{I\hspace{-0.1cm}I\hspace{-0.1cm}I}

\newtheorem{theorem}{\rm\bf Theorem}[section]
\newtheorem{proposition}[theorem]{\rm\bf Proposition}

\newtheorem{lemma}[theorem]{\rm\bf Lemma}
\newtheorem{corollary}[theorem]{\rm\bf Corollary}
\newtheorem{definition}[theorem]{\rm\bf Definition}

\theoremstyle{remark}
\newtheorem{remark}[theorem]{\rm\bf Remark}

\newcommand{\C}{{\mathbb C}}

\newcommand{\HH}{{\mathbb H}}
\newcommand{\R}{{\mathbb R}}

\newcommand{\Z}{{\mathbb Z}}

\newcommand{\tr}{{\rm Tr}}
\newcommand{\ZZ}{{\mathcal{Z}}}
\newcommand{\CC}{{\mathcal{C}}}

\newcounter{notes}%

\def\interieur#1{\mathord{\mathop{\kern 0pt #1}\limits^\circ}}

\title{Weakly almost-Fuchsian manifolds are nearly-Fuchsian}

\author{Manh-Tien Nguyen}
\address{Manh-Tien Nguyen:
University of Luxembourg, FSTM, Department of Mathematics, 
Maison du nombre, 6 avenue de la Fonte,
L-4364 Esch-sur-Alzette, Luxembourg}
\email{tien.nguyen@uni.lu}

\author{Jean-Marc Schlenker}
\address{Jean-Marc Schlenker:
University of Luxembourg, FSTM, Department of Mathematics, 
Maison du nombre, 6 avenue de la Fonte,
L-4364 Esch-sur-Alzette, Luxembourg}
\email{jean-marc.schlenker@uni.lu}

\author{Andrea Seppi}
\address{Andrea Seppi: Universit\`a di Torino, Dipartimento di Matematica ``Giuseppe Peano'',  Via Carlo Alberto 10, 10123 Torino, Italy}
\email{andrea.seppi@unito.it}

\date{}

\begin{document}

\begin{abstract}
  We show that a hyperbolic three-manifold $M$ containing a closed minimal surface with principal curvatures in $[-1,1]$ also contains nearby (non-minimal) surfaces with principal curvatures in $(-1,1)$. When $M$ is complete and homeomorphic to $S\times\mathbb{R}$, for $S$ a closed surface, this implies that $M$ is quasi-Fuchsian, answering a question left open from Uhlenbeck's 1983 seminal paper. Additionally, our result implies that there exist (many) quasi-Fuchsian manifolds that contain a closed surface with principal curvatures in $(-1,1)$, but no  closed minimal surface with principal curvatures in $(-1,1)$, disproving a conjecture from the 2000s.
\end{abstract}

\maketitle

\tableofcontents

\section{Introduction}

Embedded surfaces, and in particular minimal surfaces, play an important role in the study of hyperbolic three-manifolds. During the 1980s, the visionary works of Anderson \cite{anderson,anderson2} and Uhlenbeck \cite{uhlenbeck}, together with the groundbreaking existence results \cite{schoenyau,sacksuhlenbeck}, provided a fundamental input on minimal surfaces and quasi-Fuchsian manifolds. Moreover, the work of Uhlenbeck, together with the study of the hyperbolic Gauss map of (not necessarily minimal) surfaces initiated by Epstein \cite{epstein2,epstein}, highlighted the central role of surfaces satisfying the condition that the principal curvatures are in $(-1,1)$ --- that is, ``horospherically convex on both sides''. Since then, many advances have been obtained, particularly  from the 2000s and until very recently, see \cite{taubes,hass,sanders,huangwang,huangwang2,huangluciatarantello,huanglowe,labourie,CMN} and many others.

\subsection{Main result}

In this paper, we investigate the condition of having principal curvatures in $(-1,1)$ or in $[-1,1]$, for minimal and non-minimal surfaces. Our main result shows that, if a hyperbolic three-manifold contains a closed minimal surface with principal curvatures  in $[-1,1]$, then one can find nearby (non-minimal) surfaces with principal curvatures  in $(-1,1)$.

\begin{theorem}\label{thm:main}
  Let $\Sigma$ be a closed, orientable, two-sided, embedded minimal surface in a hyperbolic three manifold $M$, such that the principal curvatures of $\Sigma$ are in $[-1,1]$. Then any neighbourhood of $\Sigma$ in $M$ contains a closed embedded surface with principal curvatures in $(-1,1)$.
\end{theorem}

Observe that Theorem \ref{thm:main} does not make use of the global structure of $M$. That is, there is no assumption on the topology or on the completeness of $M$. 
When restricting to \emph{complete} hyperbolic three-manifolds $M$ \emph{homeomorphic to} $S\times\R$, where $S$ is a closed orientable surface of genus at least two, the following terminology is in place (from \cite{minsurf,huangwang,huanglowe}):
\begin{itemize}
\item $M$ is called \emph{nearly-Fuchsian} if it contains a closed surface homotopic to $S\times\{\star\}$ with principal curvatures in $(-1,1)$;
\item $M$ is called  \emph{almost-Fuchsian} if it contains a closed minimal surface homotopic to $S\times\{\star\}$ with principal curvatures in $(-1,1)$;
\item $M$ is called \emph{weakly almost-Fuchsian} if it contains a closed minimal surface homotopic to $S\times\{\star\}$ with principal curvatures in $[-1,1]$.
\end{itemize}
In this setting, an immediate consequence of Theorem \ref{thm:main} is:

\begin{corollary}\label{cor:weakly nearly}
  Every weakly almost-Fuchsian manifold is nearly-Fuchsian. 
\end{corollary}

Let us now discuss the implications of these results.

\subsection{A question of Uhlenbeck}

An open question since the work of Uhlenbeck (see \cite[Theorem 3.3]{uhlenbeck}) is whether a weakly almost-Fuchsian manifold $M$ is necessarily quasi-Fuchsian. This is well-known when $M$ is almost-Fuchsian, but remained open when the principal curvatures achieve the maximum value 1.  Since a nearly-Fuchsian manifold is quasi-Fuchsian (see \cite[Appendix]{uhlenbeck}, or \cite[Proposition 4.18]{elemam-seppi} in arbitrary dimensions), we prove that the answer is affirmative.

\begin{corollary}\label{cor:uhlenbeck}
  Every weakly almost-Fuchsian manifold is quasi-Fuchsian.
\end{corollary}

Let us present a few previous results around Corollaries \ref{cor:weakly nearly} and \ref{cor:uhlenbeck}. Recently, \cite[Corollary 1.4]{huanglowe} proved that if a weakly almost-Fuchsian manifold does not contain any accidental parabolic, then it is quasi-Fuchsian. Compared to their result, we provide an independent, unified proof, removing the assumption on the accidental parabolics. In fact, as a consequence of Corollary \ref{cor:uhlenbeck}, every weakly almost-Fuchsian manifold does not contain any accidental parabolic.

Forgetting for a moment the minimality condition on $\Sigma$, in \cite[Section 4]{rubinstein} Rubinstein constructed examples of closed embedded surfaces with principal curvatures in $[-1,1]$ in a hyperbolic three-manifold that has accidental parabolics. That is, a ``weakly nearly-Fuchsian'' manifold might not be quasi-Fuchsian. In particular, the conclusion of Theorem \ref{thm:main} fails if $\Sigma$ is only a closed embedded, non-minimal, surface. The recent work of Davalo (see \cite[Theorem 1.1, Theorem 1.2]{davalo}) constructed quasi-Fuchsian manifolds that are not nearly-Fuchsian.

Finally, we also remark that the class of nearly-Fuchsian manifolds is easily seen to be open inside the deformation space of complete hyperbolic structures on $S\times\R$, because a perturbation of the ambient hyperbolic structure will maintain the condition that the principal curvatures are in $(-1,1)$ for small times. Therefore, the conclusion of Corollaries \ref{cor:weakly nearly} and  \ref{cor:uhlenbeck} (and Corollary \ref{cor:big conjecture} below)  still holds true for $M$ in a neighbourhood of the space of weakly almost-Fuchsian manifolds.

\subsection{Nearly-Fuchsian that are not almost-Fuchsian}

It has been conjectured since the early 2000s (see the historical discussion below) that every nearly-Fuchsian manifold is almost-Fuchsian. Equivalently, the conjecture was that if a quasi-Fuchsian manifold $M$ contains a closed surface with principal curvatures in $(-1,1)$, then it contains a closed \emph{minimal} surface with principal curvatures in $(-1,1)$.

As a consequence of our main result, we disprove this conjecture. It is known from \cite[Theorem 4.4, Corollary 4.5]{uhlenbeck} that, for every complex structure $X$ on $S$ and every holomorphic quadratic differential $\mathfrak q$ on $(S,X)$, there is a unique weakly almost-Fuchsian manifold, but not almost-Fuchsian, such that the (unique) closed minimal surface in $M$ has first fundamental form conformal to $X$, and second fundamental form equal to the real part of $t\mathfrak q$, for a (unique) value of $t>0$. So, there is abundance of non-almost-Fuchsian, weakly almost-Fuchsian manifolds. 
\begin{corollary}\label{cor:big conjecture}
Let $M$ be a weakly almost-Fuchsian manifold that is not almost-Fuchsian. Then $M$ contains a closed surface with principal curvatures in $(-1,1)$, but no  closed minimal surface with principal curvatures in $(-1,1)$.
\end{corollary}

Indeed, from \cite{uhlenbeck} again, it is well-known that a closed minimal surface in a weakly almost-Fuchsian manifold is unique (see also \cite[Theorem 1.3]{huang-lowe-seppi} for a more general statement), and therefore a weakly almost-Fuchsian manifold cannot contain any other closed minimal surfaces with principal curvatures in $(-1,1)$.

The history of this problem is the following. In \cite{andrewsICM}, Andrews discussed applications of  geometric flows that evolve by suitably chosen symmetric functions of the principal curvatures, mostly for the study of minimal surfaces in $\mathbb S^3$, and proposed a similar approach for $\HH^3$. Citing \cite[Section 6]{andrewsICM}, ``The surfaces of interest are those for which all of the principal curvatures are less than 1 in magnitude. We can find a flow which deforms any such surface in a compact hyperbolic manifold to a minimal surface, while keeping the principal curvatures less than 1 in magnitude. Rather surprisingly, this flow is in a way the hyperbolic analogue of the one we just described for the sphere: Instead of moving with speed equal to the sum of the arctangents of the principal curvatures, we move with speed equal to the sum of the hyperbolic arctangents of the principal curvatures. [...]''. At first, the conjecture was believed to be true, see for example \cite{rubinstein,minsurf}, but it remained open since now. Our work provides a solution (in the negative) to the conjecture. 

\subsection{A partial converse}

Having established that the classes of almost-Fuchsian and nearly-Fuchsian manifolds do not agree, the question of comparing these two classes naturally arises. The space of nearly-Fuchsian manifolds is trivially contained in the space of almost-Fuchsian manifolds. We remark that a partial converse statement, namely that ``$\varepsilon$-nearly-Fuchsian manifolds'' are almost-Fuchsian, follows from analysing the quasiconformal dilatation of the hyperbolic Gauss maps (\cite{epstein2}) and from the estimates on the principal curvatures of the minimal surface in terms of the quasiconformal constant of the limit quasicircle.

\begin{theorem}\label{thm:converse}
There exists a universal constant $\varepsilon>0$ such that the following holds. If $M$ is a quasi-Fuchsian manifold containing a closed embedded surface homotopic to $S\times\{\star\}$ with principal curvatures in $(-\varepsilon,\varepsilon)$, then $M$ is almost-Fuchsian.
\end{theorem}

We stress here that $\varepsilon$ is completely independent on $M$ (it does not even depend on the genus of $S$).
In terms of geometric flows, it would be interesting to see if, under the assumption of existence of a closed surface with principal curvatures in $(-\varepsilon,\varepsilon)$, Andrews' flow by the sum of hyperbolic inverse tangents of the principal curvatures (or any other geometric flow) converges to the (unique) almost-Fuchsian minimal surface.

\subsection{Acknowledgements}
The authors are very grateful to Filippo Mazzoli for many discussions related to this work. The third author is grateful to Francesco Bonsante, Zeno Huang and Peter Smillie for interesting discussions.

A.S. was funded by the European Union (ERC, GENERATE, 101124349). Views and opinions expressed are however those of the author(s) only and do not necessarily reflect those of the European Union or the European Research Council Executive Agency. Neither the European Union nor the granting authority can be held responsible for them. A.S. is member of the national research group GNSAGA.

\section{Minimal surfaces and normal variations}

\subsection{First definitions}

Throughout the paper, $(M,h)$ will denote a hyperbolic three-manifold, and $\Sigma$ will denote an embedded surface in $M$. All manifolds and submanifolds will be orientable. Let $\nu$ be a unit  vector field normal to $\Sigma$, that can be defined globally by the orientability assumption. We start by recalling basic definitions and facts in the differential geometry of surfaces.
\begin{itemize}
\item The \emph{first fundamental form} of $\Sigma$ is the Riemannian metric $I:=\iota^*h$, where $\iota$ is the inclusion. 
\item The \emph{second fundamental form} of $\Sigma$ is the symmetric $(0,2)$-tensor $\II$ on $\Sigma$ defined by the identity:
$$\nabla^{M}_XY=\nabla^\Sigma_XY+\II(X,Y)\nu$$
where $X,Y$ are smooth vector fields tangent to $\Sigma$, and $\nabla^{M}$ and $\nabla^{\Sigma}$ are the Levi-Civita connections of $(M,h)$ and $(\Sigma,I)$ respectively. 
\item The \emph{shape operator} of $\Sigma$ is the $(1,1)$-tensor $B$ on $\Sigma$ defined by $B(X)=-\nabla^M\nu$. It is self-adjoint with respect to $I$ and satisfies the Weingarten identity
$$\II(X,Y)=I(B(X),Y)~.$$
The \emph{principal curvatures} are the eigenvalues of $B$.
\item The \emph{third fundamental form} of $\Sigma$ is the symmetric $(0,2)$-tensor $\III$ on $\Sigma$ defined by $\III(X,Y)=I(B(X),B(Y))$.
\end{itemize}

The first fundamental form and the shape operator satisfy the Gauss-Codazzi equations, namely $K_I=-1-\det B$ (Gauss' equation) and $(\nabla^\Sigma_XB)(Y)=(\nabla^\Sigma_YB)(X)$ (Codazzi equation). Conversely, every pair $(I,B)$ on a closed surface $\Sigma$ satisfying the Gauss-Codazzi equations can be realised as the first fundamental form and shape operator of an immersion of $\widetilde \Sigma$ into $\HH^3$, which is unique up to post-composition with an isometry of $\HH^3$ and equivariant with respect to a representation of $\pi_1(\Sigma)$ in the group of isometries of $\HH^3$.

An embedded surface $\Sigma$ in $M$ is \emph{minimal} if the trace of its shape operator vanishes identically. Equivalently, the principal curvatures of $\Sigma$ are opposite at every point of $\Sigma$. We will denote the principal curvatures of a minimal surface by $\lambda^+\geq 0$ and $\lambda^-\leq 0$, so that $\lambda^+=-\lambda^-$. 

Recall also that the norm of the second fundamental form of $\Sigma$ is defined as 
\begin{equation}\label{defi normII}
\|\II\|=\sqrt{\tr(B^2)}~.
\end{equation}
 Hence, if $\Sigma$ is minimal, then $\|\II\|^2=(\lambda^+)^2+(\lambda^-)^2=2(\lambda^+)^2$. In particular the principal curvatures, at a given point $p$, are $\lambda^+(p)=1$ and $\lambda^-(p)=-1$ (resp. $\lambda^+(p)<1$ and $\lambda^-(p)>-1$) if and only if $\|\II\|^2(p)=2$ (resp. $\|\II\|^2(p)<2$).

\subsection{Normal variations}

The approach to the proof of Theorem \ref{thm:main} is the following. Assuming that $\Sigma$ is a closed minimal surface in $M$ with principal curvatures in $[-1,1]$, we aim to construct a smooth perturbation of $\Sigma$ that decreases the positive principal curvature of $\Sigma$, and increases the negative principal curvature, for small times, so that the principal curvatures of ``nearby'' surfaces are in $(-1,1)$. For this purpose, we introduce and study \emph{normal} variations with respect to a given function $f$.

Let $\Sigma$ be an embedded   surface in  $M$. Let $f\in C^\infty(\Sigma)$. Then we define
$\iota_f:\Sigma\to M$ as 
$$\iota_f(p)=\exp_p(f\nu(p))~.$$

\begin{lemma}\label{lemma:embedded}
Let $\Sigma$ be a closed embedded surface in a hyperbolic three-manifold $M$ and let $f\in C^\infty(\Sigma)$. There exists $\epsilon>0$ such that $\iota_{t\!f}$ is an embedding for $t\in(-\epsilon,\epsilon)$.
\end{lemma}

\begin{proof}
Consider the map $\mathcal E:\Sigma\times\R\to M$ sending $(p,t)$ to $\exp_p(t\nu(p))$. For every $p\in\Sigma$, the differential of $\mathcal E$ at $(p,0)$ is easily checked to be an isomorphism, hence there exists $\epsilon>0$ such that $\mathcal E|_{\Sigma\times(-\epsilon,\epsilon)}$ is a local diffeomorphism. Now, since $\mathcal E$ is injective on $\Sigma\times\{0\}$, and locally injective around every point of $\Sigma\times\{0\}$, it follows from compactness of $\Sigma$ that, up to taking a smaller $\epsilon$,  $\mathcal E|_{\Sigma\times(-\epsilon,\epsilon)}$ is injective (see for instance \cite[Lemma 3.6]{choudhury-mazzoli-seppi}), and thus a diffeomorphism onto its image. This concludes the proof, because $\iota_{t\!f}(p)=\mathcal E(p,tf(p))$ and the map $p\mapsto(p,tf(p))$ is an embedding of $\Sigma$ in $\Sigma\times\R$.
\end{proof}

Let us denote by $\Sigma_f$ the image of $\iota_f$. Observe that $\Sigma_0=\Sigma$ and $\iota_0$ is the inclusion of $\Sigma$ in $M$. When $t$ is sufficiently small, so that $\iota_{t\!f}$ is an embedding as in Lemma \ref{lemma:embedded}, we denote by $I_{t\!f}=(\iota_{t\!f})^*h$ the first fundamental form, by $\II_{t\!f}$ the second fundamental form, and  by $B_{t\!f}$ the shape operator, of $\Sigma_{t\!f}$, pulled-back to $\Sigma$ via $\iota_{t\!f}$. The second fundamental form and shape operator are computed with respect to the unit normal vector field that continuously extends $\nu$.

\begin{lemma}
Let $\Sigma$ be a closed embedded surface in a hyperbolic three-manifold $M$, and let $f\in C^\infty(\Sigma)$. Then
\begin{equation}\label{eq:var B}
\left.\frac{d}{dt}\right|_{t=0} B_{t\!f}=\mathrm{Hess}^\Sigma \!f+f(B^2-\mathbbm{1})~,\end{equation}
where $\mathrm{Hess}^\Sigma \!f$ is the $(1,1)$-Hessian of $\Sigma$, defined by $I(\mathrm{Hess}^\Sigma \!f(v),w)=(\nabla_v^\Sigma df)(w)$, and $\mathbbm 1$ denotes the identity operator on $T\Sigma$.

In particular, if $\Sigma$ is minimal and $p$ is a point where $\|\II\|^2(p)=2$, then
\begin{equation}\label{eq:var B at Z}
\left(\left.\frac{d}{dt}\right|_{t=0} B_{t\!f}\right)_p=(\mathrm{Hess}^\Sigma \!f)_p~.
\end{equation}
\end{lemma}

\begin{proof}
The following formula for the variation of the second fundamental form holds, see \cite[Theorem 3-15]{andrewsJDG}:
\begin{equation}\label{eq:var II}
\left.\frac{d}{dt}\right|_{t=0} \II_{t\!f}=\nabla^\Sigma df-f(I+\III)~.
\end{equation}
We shall combine \eqref{eq:var II} with the well-known identity
$$\left.\frac{d}{dt}\right|_{t=0} I_{t\!f}=-2f\II~,$$
which implies 
\begin{equation}\label{eq:var I-1}
\left.\frac{d}{dt}\right|_{t=0} I^{-1}_{t\!f}=2fI^{-1}\II I^{-1}=2fBI^{-1}~.
\end{equation}
So, differentiating $B_{t\!f}=I^{-1}_{t\!f}\II_{t\!f}$ and using \eqref{eq:var II} and \eqref{eq:var I-1}, we obtain the desired formula \eqref{eq:var B}. 

For the second part, if $\Sigma$ is minimal, then $\tr(B)=0$, and if moreover the principal curvatures are $\pm 1$ at $p$, then $\det(B)=-1$ at $p$. By the Cayley-Hamilton Theorem, $B^2-\tr(B)B+\det(B)\mathbbm 1=0$, which implies $B^2=\mathbbm 1$ at the point $p$. Hence the last term in \eqref{eq:var B} vanishes, and this concludes the proof.
\end{proof}

In the following lemma, we compute the variation of the principal curvatures of $\Sigma_{t\!f}$ at $p$, again in the hypothesis that $\Sigma$ is minimal and that $p$ is such that $\|\II\|^2(p)=2$. For $p$ outside the zeros of $B$ and for small $t$, we denote by $\lambda^+_{t\!f}(p)$ (resp. $\lambda^-_{t\!f}(p)$) the positive (resp. negative) eigenvalue of $B_{t\!f}$.

Observe that $\lambda^+_{t\!f}(p)$ and $\lambda^-_{t\!f}(p)$ are smooth functions of $t$ (and $p$) as long as they are different. Indeed, $B_{t\!f}$ depends smoothly on $t$ and $p$, and the eigenvalues of a matrix are smooth functions in the entries of the matrix, as long as the discriminant does not vanish. So, when $\Sigma$ is minimal, $\lambda^+_{t\!f}(p)$ and $\lambda^-_{t\!f}(p)$ are smooth functions of $t$ outside of the zeros of $B$, for small $t$.

We introduce here additional notation. We denote by $(e_+,e_-)$ a smooth local frame of unit eigenvectors of $B$, where $e_\pm$ is an eigenvector of $\lambda^\pm$, defined outside the zeros of $B$. It is uniquely determined up to changing sign to $e_\pm$, which does not affect the following statements.

\begin{lemma}\label{lemma:variation principal curv}
Let $\Sigma$ be a closed embedded minimal  surface in a hyperbolic three-manifold $M$, and let $f\in C^\infty(\Sigma)$. Suppose $\|\II\|^2(p)=2$.
Then
$$\left.\frac{d}{dt}\right|_{t=0}\lambda^\pm_{t\!f}(p)=(\nabla^\Sigma df)(e_\pm(p),e_\pm(p))~.$$
\end{lemma}
\begin{proof}
Let us write, in the frame  $(e_+,e_-)$ around a point $p$ such that $\|\II\|^2(p)=2$,
$$B_{t\!f}=\begin{pmatrix} \alpha(t) & \beta(t) \\ \gamma(t) & \delta(t) \end{pmatrix}~.$$
Now, differentiating the formula $\lambda^\pm_{t\!f}=(1/2)(\tr(B_{t\!f})\pm\sqrt{\tr(B_{t\!f})^2-4\det(B_{t\!f})})$ and using that $\tr(B)_p=0$ and $\det(B)_p=-1$, we obtain
\begin{align*} \left.\frac{d}{dt}\right|_{t=0}\lambda^\pm_{t\!f}(p)&=\frac{1}{2}\left(\tr\left(\frac{d}{dt} B_{t\!f}\right)_{\!\!p} \pm \tr\left(B\frac{d}{dt} B_{t\!f}\right)_{\!\!p} \right) \\
& =\frac{1}{2}\left(\left(\left.\frac{d}{dt}\right|_{t=0}\alpha(t)+\left.\frac{d}{dt}\right|_{t=0}\delta(t)\right)\pm \left(\left.\frac{d}{dt}\right|_{t=0}\alpha(t)-\left.\frac{d}{dt}\right|_{t=0}\delta(t)\right)\right)\end{align*}
That is, 
$$\left.\frac{d}{dt}\right|_{t=0}\lambda^+_{t\!f}(p)=\left.\frac{d}{dt}\right|_{t=0}\alpha(t) \qquad\text{and}\qquad \left.\frac{d}{dt}\right|_{t=0}\lambda^-_{t\!f}(p)=\left.\frac{d}{dt}\right|_{t=0}\delta(t)~.$$
Combining with \eqref{eq:var B at Z} and the fact that $(e_+,e_-)$ is an orthonormal frame, this concludes the proof.
\end{proof}
In order to prove Theorem \ref{thm:main}, it will be essential to understand the locus of $p$ in $\Sigma$ where the principal curvatures are equal to $1$ and $-1$. We thus introduce the following notation.

\begin{definition}\label{defi ZZ}
  Let $\Sigma$ be a closed minimal surface in a hyperbolic three-manifold $M$ such that $\|\II\|^2\leq 2$. We denote 
  $$\ZZ:=\{p\in\Sigma\,|\,\|\II\|^2(p)= 2\}~.$$
\end{definition}

In Section \ref{sec:ZZ} we investigate the structure of $\ZZ$. Before that, we introduce special charts for $\Sigma$ around points of $\ZZ$, which will be very useful for the analysis.

\subsection{Half-translation structures}

Given a minimal surface $\Sigma$ in $M$, it is well-known (from \cite{hopf}, see also \cite{lawson,hopfbook,tromba}) that the second fundamental form of $\Sigma$ is the real part of a holomorphic quadratic differential, that we will denote by $\mathfrak q$, on $(\Sigma,[I])$, where $[I]$ denotes the complex structure on $\Sigma$ associated with the first fundamental form $I$. By the holomorphicity condition, the zeros of $\mathfrak q$ are isolated points. For any point $p$ in $\Sigma\setminus \{\mathfrak q=0\}$, one can find a local chart $z:U\to\C$ around $p$ such that $\mathfrak q$ is expressed as $dz^2$. Indeed, if $w$ is an arbitrary local complex chart for $(\Sigma,[I])$, then $\mathfrak q$ is expressed as $\xi(w)dw^2$, and it suffices to choose $z$ a determination of the square root of $\xi(z)$, which is possible as long as $\xi$ does not vanish and $U$ is simply connected.

This procedure endows $\Sigma\setminus \{\mathfrak q=0\}$ with an atlas with values in $\C$ such that the change of coordinates, by an elementary computation, are of the form $z\mapsto \pm z + c$, for $c\in\C$. We shall call any chart of this form a \emph{flat coordinate} for $\Sigma$. That is, in the language of $(G,X)$-structures, $\Sigma\setminus \{\mathfrak q=0\}$ is naturally endowed with a half-translation structure. 

In particular, since the transformations of the form $z\mapsto \pm z + c$ are isometries for the flat metric $|dz|^2=dx^2+dy^2$, $\Sigma\setminus \{\mathfrak q=0\}$ is also naturally endowed with a flat Riemannian metric compatible with the complex structure of $[I]$. We will call this flat metric the \emph{Euclidean metric}.  We remark that the Euclidean metric has a conical singularity with angle an integer multiple of $\pi$ around the zeros of $\mathfrak q$, but we will never need to analyse the structure around the zeros in this work. 

The following elementary lemma permits to express the embedding data of $\Sigma$ in a flat coordinate.

\begin{lemma} \label{lemma:coshgordon}
Let $\Sigma$ be an embedded minimal surface in a hyperbolic three-manifold $M$, and let $z:U\to\C$ be a flat coordinate. Then, in $U$, $I=e^{2u}(dx^2+dy^2)$, $\II=dx^2-dy^2$ and $B=e^{-2u}(dx\otimes\partial_x - dy\otimes\partial_y)$, where $z=x+iy$ and $u:U\to\R$ is a smooth function solving
\begin{equation}\label{eq:coshgordon}
\Delta u=2\cosh(2u)~,
\end{equation} 
for $\Delta=\partial^2/\partial x^2+\partial^2/\partial_y^2$ the Laplace operator on $\R^2$.

Moreover, $\|\II\|^2=2e^{-4u}$. That is, if $\Sigma$ satisfies $\|\II\|^2\leq 2$ then $u\geq 0$, and if $p\in\ZZ$ then $u(p)=0$.
\end{lemma}
\begin{proof}
The expressions of $I$, $\II$ and $B$, and the ``moreover'' part, follow immediately from the definitions. The $\cosh$-Gordon equation \eqref{eq:coshgordon} is then equivalent to Gauss' equation. Indeed, the curvature of $I=e^{2u}(dx^2+dy^2)$ equals $-e^{-2u}\Delta u$, so Gauss' equation $K_I=-1-\det B$ is equivalent to $\Delta u=e^{2u}(1+e^{-4u})=2\cosh(2u)$.
\end{proof}

\section{Maxima of the principal curvatures} \label{sec:ZZ}

In this section, we study the locus $\ZZ$ for closed minimal surfaces in $M$ satisfying the condition $\|\II\|^2\leq 2$. The two main results of the section are Proposition \ref{prop:points and curves} and Proposition \ref{prop:no geodesic}.

\subsection{Points and simple closed curves}

The following result gives a first important constraint on the structure of $\ZZ$. 

\begin{proposition}\label{prop:points and curves}
Let $\Sigma$ be a closed minimal surface in a hyperbolic three-manifold $M$ such that $\|\II\|^2\leq 2$. Then $ \ZZ$ is the disjoint union of a finite number of points and smooth simple closed curves.
\end{proposition}

\begin{proof}
 Let $p\in\ZZ$ and let $z:U\to\C$ be a flat coordinate, for $U$ a neighbourhood of $p$. By Lemma \ref{lemma:coshgordon}, $\ZZ \cap U$ is, in the coordinate $z$, the zero set of a function $u:\Omega\to\R$, for $\Omega$ an open set in $\C$, solving the $\cosh$-Gordon equation \eqref{eq:coshgordon}. By elliptic regularity applied to \eqref{eq:coshgordon}, $u$ is real analytic (\cite{petrowsky,morrey}). Up to applying a translation to $z$, we can assume that $p$ corresponds to the origin, and up to a further rotation, we can assume that $u_{yy}(0)> 0$, since $\Delta u=2\cosh(2u)>0$. (Observe that this rotation will change the expression for $\II$, which in the original coordinate $z$ was $dx^2-dy^2$, by conjugation, but this will not affect the rest of the argument.) In particular, $u(0,y)$ is not identically zero. 
 
We now apply Lojasiewicz's Structure Theorem for analytic varieties, see \cite[Theorem 6.3.3]{krantz-parks}. In dimension 2, this states that, up to restricting the neighbourhood $\Omega$, $\{z\in\Omega\,|\,u(z)=0\}$ is one of the following:
\begin{enumerate}
\item the origin alone;\label{item:origin}
\item a smooth curve, homeomorphic to an interval, represented as the graph of a real analytic function of the $x$-coordinate, and containing the origin;\label{item:curve}
\item a union of finitely many ``branches'', namely pairwise disjoint smooth curves, homeomorphic to intervals, represented as the graphs of real analytic functions of the $x$-coordinate, all containing the origin in their closure.\footnote{In fact, although not necessary for our argument here, it is known by \cite[Corollary 2]{sullivan} that in case (3) there is an even number of branches.}
\end{enumerate}
We now claim that, inside such a $\Omega$, $\{u=0\}$ is in fact of the form \eqref{item:origin} or \eqref{item:curve}. Indeed, by Lemma \ref{lemma:coshgordon} again, the set of zeros of $u$ consists entirely of minima of $u$, and therefore it is contained inside the set $\{z\in\Omega\,|\,u_y(z)=0\}$. Since $u_{yy}(0)\neq 0$, by the implicit function theorem $\{u_y=0\}$ is a smooth curve, which is locally a graph over the $x$-axis. Hence $\{u=0\}$ is contained inside such smooth curve, and, by the trichotomy above, it is either an isolated point or a smooth curve. 

This shows that every point of $\ZZ$ has a neighbourhood which consists either of an isolated point or of a smooth curve. Since $\Sigma$ is compact and $\ZZ$ is closed in $\Sigma$, being the level set of the  continuous (in fact, real analytic) function $\|\II\|^2$, $\ZZ$ is compact, and therefore it consists of finitely many points and simple closed curves.
\end{proof}

\subsection{There are no geodesics in $\ZZ$}

The next main result of the section, Proposition \ref{prop:no geodesic}, roughly shows that no simple closed curve in $\ZZ$ is a geodesic for the first fundamental form. (Actually, as explained in Lemma \ref{lemma first or euclidean}, being geodesic for the first fundamental form is equivalent to being geodesic for the Euclidean metric.) This will be of fundamental importance in Section \ref{sec:function}. 

All the proofs of this subsection work, without further modification, for a curve contained in the critical set of the function $\|\II\|^2$, so we will state the results in this setting. So, given a function $f:\Sigma\to\R$, let us denote $$\mathrm{Crit}(f)=\{p\in\Sigma\,|\,df_p=0\}~.$$
Then let us denote 
  $$\CC:=\mathrm{Crit}(\|\II\|^2)\setminus \{p\in\Sigma\,|\,\|\II\|^2(p)=0\}~.$$
  Clearly, $\ZZ\subset\CC$. 
  
 \begin{remark}
  Taking away the zeros of $\|\II\|^2$ does not change anything in the conclusions. In fact,  the function $\|\II\|^2$ is a real analytic function on $\Sigma$, as a consequence of the analyticity of minimal surfaces and of the definition of the norm of the second fundamental form in \eqref{defi normII}. Since the zeros of  $\|\II\|^2$ are the zeros of the holomorphic quadratic differential $\mathfrak q$ whose real part is $\II$, they are isolated in $\Sigma$. By the properties of analytic functions, if $p$ is a zero, then there is a small neighbourhood $U$ of $p$ such that every critical point of $\|\II\|^2$ inside $U$ is a zero.\footnote{\label{note1}More precisely, there is a stronger version of Sard's Lemma for analytic functions (\cite{MR393394}): if $f$ is a real analytic function and $K$ is a compact set, then $f(K\cap\mathrm{Crit}(f))$ is finite.} Hence, zeros of $\|\II\|^2$ are isolated in the critical set. We removed them from the definition of $\CC$ only to keep the statements simpler.
\end{remark}

Before the proof of Proposition \ref{prop:no geodesic}, we need some preparation. First, the following well-known fact, which essentially follows from the ideas in \cite{epstein}, will be useful in several instances.

\begin{proposition}\label{prop:complete is proper}
If $\iota:\Sigma_0\to\HH^3$ is an immersion of a surface such that  the first fundamental form is complete and $\|\II\|^2\leq 2$, then $\iota$ is a proper embedding, and $\Sigma_0$ is diffeomorphic to $\R^2$.
\end{proposition}

The proof can be found in \cite[Proposition 4.15]{elemam-seppi} under the assumption that $\|\II\|^2< 2$. The arguments extend under the hypothesis $\|\II\|^2\leq 2$, see \cite[Appendix A]{huang-lowe-seppi}.  

Second, we observe that the notion of geodesic for the first fundamental form and for the Euclidean metric is equivalent, for curves in $\CC$.

\begin{lemma}\label{lemma first or euclidean}
Let $\Sigma$ be an embedded minimal surface in a hyperbolic three-manifold $M$ such that $\|\II\|^2\leq 2$, and let $\gamma:I\to\CC$ be a regular curve. Then parallel transports along $\gamma$ for the first fundamental form $I$ and for the Euclidean metric coincide. In particular, $\gamma$ is a geodesic for $I$ if and only if it is a geodesic for the Euclidean metric.
\end{lemma}
\begin{proof}
It suffices to recall from Lemma \ref{lemma:coshgordon} that, in a flat coordinate $z$, the first fundamental form is expressed as $e^{2u}(dx^2+dy^2)$, the squared norm of the second fundamental form is $2e^{-4u}$. So the points of $\CC$ are the critical points of $u$, which means that all first derivatives of the metric tensor vanish at $\CC$, and therefore the Christoffel symbols vanish. 
\end{proof}

Third, we need to study some particular solutions of the equation \eqref{eq:coshgordon}, invariant by translations. The proof of Proposition \ref{prop:no geodesic} will then go by contradiction: assuming that a curve in $\CC$ can be a geodesic, this will imply that the minimal surface must coincide with the minimal surface associated to one of these translation-invariant solutions, and this will easily lead to a contradiction.

\begin{lemma}\label{lemma:ODE}
Let $\ell$ be a line in $\C$ and let $v_0\geq 0$. There exists a maximal solution $v:\Omega_\ell\to\R$, where $\Omega_\ell=\{z\in\C\,|\,d(z,\ell)<\delta\}$, to the equation $\Delta v=2\cosh(2v)$ such that $v|_{\ell}=v_0$ and $dv|_\ell=0$. Moreover, $I=e^{2v}(dx^2+dy^2)$ and $\II=dx^2-dy^2$ are the first and second fundamental forms of a minimal proper embedding of $\Omega_\ell$ in $\HH^3$. 
\end{lemma}

\begin{proof}
Applying a translation and a rotation to $\C$, we can assume that $\ell=\{x=0\}$. (As for the proof of Proposition \ref{prop:points and curves}, this operation conjugates the form of $\II$ by a rotation, but does not affect the rest of the argument.) The desired solution $v$ then has the form 
$$v(x,y)=g(x)$$
 where $g:(-\delta,\delta)\to\R$ is the unique maximal solution of the ODE 
 \begin{equation}\label{eq:ode}
g''(x)=2\cosh(2g(x))
\end{equation}
with Cauchy data $g(0)=v_0$ and $g'(0)=0$. 

First, we claim that $\delta<+\infty$. By contradiction, if $\delta=+\infty$, then $e^{2v}(dx^2+dy^2)$ would be a conformal metric on $\C$ of curvature bounded above by $-1$, since by Lemma \ref{lemma:coshgordon} the cosh-Gordon \eqref{eq:coshgordon} equation is equivalent to the Gauss' equation, and therefore the curvature equals $-1-e^{-4u}\leq -1$. But such a conformal metric on $\C$ does not exist.\footnote{This is an immediate consequence of Ahlfors-Schwarz Lemma \cite{ahlfors}. If such a conformal metric exists, then for every $R>0$ its conformal factor would be smaller than the conformal factor $4/(R^2-|z|^2)$ of the Poincar\'e metric on the ball of radius $R$. Letting $R$ tend to infinity, such a conformal factor would be zero, giving a contradiction.}
Incidentally, we observe that $g\geq 0$, as a consequence of the conditions $g(0)>0$, $g'(0)=0$ and $g''>0$.

So, $v$ is defined on the domain $\Omega=\{|x|<\delta\}$, which is the maximal domain of definition. Since the domain is simply connected, by the fundamental theorem of surfaces one obtains a minimal immersion of $\Omega$ in $\HH^3$ with prescribed first and second fundamental forms. Its principal curvatures are in $[-1,1]$ since $v\geq 0$. In order to conclude the proof, by Lemma \ref{prop:complete is proper}, it suffices to show that the first fundamental form $I=e^{2v}(dx^2+dy^2)$ is complete. Since the vertical translations $(x,y)\mapsto(x,y+y_0)$ are isometries, it is enough to check that the curve $\eta=(\delta,\delta)\times\{0\}$ has infinite length. Its length is:
\begin{equation}\label{eq:completeness}
\mathrm{length}(\eta)=\int_{-\delta}^{\delta}e^{g(x)}dx=2\int_{0}^{\delta}e^{g(x)}dx~,
\end{equation}
and we want to show that this integral diverges. For this purpose, we multiply both sides of \eqref{eq:ode} by $g'(x)$ and integrate, to obtain
\begin{equation}\label{eq:ode2}
g'(x)^2=2\sinh(2g(x))-2\sinh(2v_0)~.
\end{equation}
A simple chain of inequalities gives:
\begin{equation}\label{eq:ode3}
e^{2g(x)}>\frac{\sinh(2g(x))}{2}\geq \frac{\sinh(2g(x))}{2}-\frac{\sinh(2v_0)}{2}=\frac{g'(x)^2}{4}~,
\end{equation}
that is, $e^{g(x)}>g'(x)/2$. Inserting this into \eqref{eq:completeness} we have
$$\mathrm{length}(\eta)>\int_{0}^{\delta}g'(x)dx=-{v_0}+\lim_{x\to\delta^-}g(x)~.$$
But $g$ must diverge as $x$ approaches $\delta$. Indeed, if $g$ had a finite limit as $x\to\delta^-$, then  $g'$ would too have a finite limit by \eqref{eq:ode3}, and by the Escape Lemma this would contradict that $\delta$ is the maximal time of existence. This shows that the first fundamental form is complete and thus concludes the proof.
\end{proof}

\begin{remark}
Let $\ell$ be the vertical axis in $\C$, let $v:\{|x|<\delta\}\to\R$ the solution provided by Lemma \ref{lemma:ODE} with initial condition $v|_\ell=v_0$ and $dv_\ell=0$, and consider the minimal immersion having first and second fundamental forms $I=e^{2v}(dx^2+dy^2)$ and $\II=dx^2-dy^2$. If $v_0=0$, then  the image of the immersion is the so-called ``mouse-hole'' minimal surface, which is a complete minimal surface invariant under a one-parameter group of parabolic isometries. If $v_0>0$, then the minimal surface is invariant by a one-parameter group of loxodromic isometries. Of course there are similar solutions for $v<0$, whose output is a minimal catenoid, which is not embedded anymore. The line $\ell$ is a line of curvature, and is mapped in $\HH^3$ to a horocycle (if $v_0=0$), a planar curve equidistant from a geodesic (if $v_0>0$), or a geodesic circle (if $v_0<0$),  by the corresponding minimal immersion. See \cite{docarmo-dajczer} for more details, and \cite[Section 4]{Zelnikov} for nice illustrations. 
\end{remark}

We are now ready to prove Proposition \ref{prop:no geodesic}.

\begin{proposition}\label{prop:no geodesic}
Let $\Sigma$ be a closed minimal surface in a hyperbolic three-manifold $M$ such that $\|\II\|^2\leq 2$. Then no non-constant curve $\gamma:(a,b)\to \CC$ is a geodesic with respect to $I$.  
\end{proposition}

\begin{proof}
  Suppose by contradiction that $\gamma:(a,b)\to \CC$ is a geodesic. By Lemma \ref{lemma first or euclidean}, $\gamma$ is a geodesic also for the Euclidean metric, and therefore, in a flat coordinate $z:U\to\C$ (restricting the interval $(a,b)$ if necessary, so that $\gamma(a,b)\subset U$), the image of $\gamma$ is contained in a line $\ell\subset\C$. 
  
 By Lemma \ref{lemma:coshgordon}, the first fundamental form $I$ of $\Sigma$ equals $e^{2u}(dx^2+dy^2)$, and the second fundamental form equals $dx^2-dy^2$, where $u$ solves $\Delta u=2\cosh(2u)$ and, for every $p$ in the image of $\gamma$, $du(p)=0$. Hence $u$ equal to a constant, say $v_0$, on the image of $\gamma$. From the ``moreover'' part of Lemma  \ref{lemma:coshgordon}, $v_0\geq 0$. By the Cauchy-Kovalevskaya theorem, up to restricting the open set $U$, $u|_U=v|_U$, where $v$ is the solution constructed in Lemma \ref{lemma:ODE}. 
 
 Now, let $\pi:\widetilde \Sigma\to\Sigma$ the universal cover of $\Sigma$. Then $\pi^*I$ and $\pi^*\II$ are the first and second fundamental form of a complete minimal immersion in $\HH^3$, which is a proper embedding by Proposition \ref{prop:complete is proper}. By a minor abuse of notation, we denote such properly embedded minimal surface again by $\widetilde\Sigma$. Likewise, from Lemma \ref{lemma:ODE}, $e^{2v}(dx^2+dy^2)$ and $dx^2-dy^2$ on $\Omega_\ell$ are the first and second fundamental form of a proper minimal embedding in $\HH^3$. We denote by $\Sigma_0$ the image of such minimal embedding. 
 
 We claim that, up to applying an isometry of $\HH^3$, $\widetilde \Sigma=\Sigma_0$. Indeed, if $\widetilde U$ denotes a lift of $U$ to $\widetilde\Sigma$, the previous paragraph shows that the embedding data of $\widetilde\Sigma$, on a flat coordinate $z:\widetilde U\to\C$, agrees with the embedding data of $\Sigma_0$. Hence, by the uniqueness part of the fundamental theorem of surfaces in $\HH^3$, $\widetilde \Sigma$ and $\Sigma_0$ agree on a nonempty open set, after applying an isometry of $\HH^3$ to one of the two. The conclusion of the claim that $\widetilde \Sigma=\Sigma_0$ then follows by a standard analytic continuation argument. Let $S$ be the set of points $p\in\widetilde\Sigma$ such that $p\in\Sigma_0$, $T_p\widetilde\Sigma=T_p\Sigma_0$, and such that $\widetilde\Sigma$ and $\Sigma_0$, when represented locally as graphs over $T_p\widetilde\Sigma=T_p\Sigma_0$, have the same partial derivatives of every order. Then $S$ is nonempty because it contains $\widetilde U$, closed by smoothness of $\widetilde\Sigma$ and $\Sigma_0$, and open by analyticity of the solutions of the minimal surface equation. Since $\widetilde\Sigma$ is connected, $S=\widetilde\Sigma$, hence $\widetilde \Sigma\subset\Sigma_0$. But since $\widetilde \Sigma$ is properly embedded, $\Sigma_0=\widetilde\Sigma$.
 
This now provides a contradiction. On the one hand, by construction $\Sigma_0$ has no umbilical points (i.e. points where then second fundamental form vanishes), because its second fundamental form has globally the expression $dx^2-dy^2$ on $\Omega_\ell$. On the other hand, $\widetilde\Sigma$ must contain umbilical points, because its second fundamental form is the lift of the second fundamental form of $\Sigma$, and every holomorphic quadratic differential on a closed Riemann surface of genus at least $2$ vanishes at certain points. 
  \end{proof}

\section{Functions decreasing the curvature} \label{sec:function}

In this section we prove a crucial technical point in order to prove Theorem \ref{thm:main}, namely the following statement.

\begin{proposition}
Let $\Sigma$ be a closed minimal surface in a hyperbolic three-manifold $M$ such that $\|\II\|^2\leq 2$. Then there exists a smooth function $f:\Sigma\to\R$ such that, for every $p$ in $\ZZ$, 
\begin{equation}\label{eq:cond signs} 
\nabla^\Sigma df(e_+(p),e_+(p))=-1\qquad\text{and}\qquad \nabla^\Sigma df(e_-(p),e_-(p))=1~.
\end{equation}
\end{proposition}

\subsection{Set-up of the proof}

We start by an easy lemma showing that, for points of $\ZZ$, the Hessian for the first fundamental form is essentially the Euclidean Hessian in a flat coordinate.

\begin{lemma} \label{lemma:euclidean hessian}
Let $\Sigma$ be an embedded minimal surface in a hyperbolic three-manifold $M$ such that $\|\II\|^2\leq 2$. Let $p\in\ZZ$, let $z:U\to\C$ be a flat coordinate around $p$, and let $f:U\to\R$ be a smooth function. Then, in the basis $(e_+(p),e_-(p))$, 
$$(\nabla^\Sigma df)_p=D^2(f\circ z^{-1})_{z_0}=\begin{pmatrix} (f\circ z^{-1})_{xx}(z_0) & (f\circ z^{-1})_{xy}(z_0) \\ (f\circ z^{-1})_{xy}(z_0) & (f\circ z^{-1})_{yy}(z_0)\end{pmatrix}$$ 
where $z_0=z(p)$ and $z=x+iy$.
\end{lemma}
\begin{proof}
By definition of the flat coordinate, the unit eigenvalues of $B$ at $p$ correspond, up to a sign that does not affect the conclusion, to the vectors $\partial_x$ and $\partial_y$ in $\R^2$. Now, 
as already observed in the proof of Lemma \ref{lemma first or euclidean}, if $p\in\ZZ$ then the Christoffel symbols of $I$ vanish at $p$ in the flat coordinate $z$, by Lemma \ref{lemma:coshgordon}. So the Riemannian Hessian (as a symmetric $(0,2)$-tensor) computed with respect to $I$ is simply given by the second derivatives of $f$ at $p$, that is, it coincides with the Hessian for the Euclidean metric. 
\end{proof}

In order to construct the function $f$, we will work around each connected component of $\ZZ$, and then glue with bump functions. Let us introduce some notation. Using Proposition \ref{prop:points and curves}, we denote
$$\ZZ=\bigsqcup_{i=1}^N \ZZ_{i}$$
where each $\ZZ_i$ is either a point or a simple closed curve. Also, let $r>0$ be such that, denoting $N_r(X)$ the neighbourhood of points at distance less than $r$ from $X$ (say, for the first fundamental form), the open sets $V(\ZZ_i):=N_r(\ZZ_i)$ are pairwise disjoint, are tubular neighbourhoods of $\ZZ_i$, and do not contain any zero of the holomorphic quadratic differential $\mathfrak q$ whose real part is $\II$. 

We also introduce a bump function that we will use repeatedly.

\begin{definition}\label{defi:bump}
  Let us denote by $\phi:\R_{\geq 0} \to [0,1]$ a smooth, non-increasing function such that $\phi(t)=1$ for $t\in [0,r/2]$, and $\phi(t)=0$ for $t\in [r/2,\infty)$.
\end{definition}

Now, we will divide the construction into several cases. 

\subsection{The easiest cases}\label{subsec:easy case}
Let us start by the simplest case, namely when $\ZZ_i$ is a point.

\begin{lemma} \label{lm:point}
  Suppose that $\ZZ_i=\{p_0\}$ is  an isolated point in $\ZZ$. Then there exists a smooth function $f:\Sigma\to \R$, supported in $V_i$, satisfying \eqref{eq:cond signs} at $p_0$.
\end{lemma}

\begin{proof}
Let $z=x+iy:U\to\C$ be a flat coordinate around $p_0$, let $F:\C\to\R$ be the function $F(x,y)=(-x^2+y^2)/2$, and let $f:\Sigma\to\R$ be defined as $f(p)=\phi(d(p,p_0))F(z(p))$. Then $f$ is supported in the $r$-neighbourhood of $p_0$, and it satisfies \eqref{eq:cond signs} by Lemma \ref{lemma:euclidean hessian}.
\end{proof}

Let us now move on to the situation where $\ZZ_i$ is a simple closed curve. In this situation, let $\widetilde V_i$ denote the universal cover of the tubular neighbourhood $V_i$, on which the fundamental group $\pi_1(V_i)\cong \Z$ acts by deck transformations. We fix a developing map $\mathrm{dev}:\widetilde V_i\to\C$ for the half-translation structure of $\Sigma\setminus \{\mathfrak q=0\}$, namely a local diffeomorphism from $\widetilde V_i$ to $\C$ such that $\mathrm{dev}^*(dz^2)=\pi^*\mathfrak q|_{V_i}$, where $\pi:\widetilde V_i\to V_i$ is the universal covering map. Such a developing map is unique up to post-composition with a transformation of the form $z\mapsto \pm z+c$, and it is equivariant with respect to the holonomy morphism $\rho$ from $\Z$ to the group of half-translations. Namely, denoting $\eta$ the generator of $\pi_1(V_i)\cong \Z$, $\mathrm{dev}\circ \eta^n=\rho(\eta^n)\circ \mathrm{dev}=\rho(\eta)^n\circ \mathrm{dev}$. 

In this setting, our goal is to construct a $\pi_1(V_i)$-invariant function $\widetilde f$ on $\widetilde V_i$, such that, locally, $\widetilde f|_U=F\circ\mathrm{dev}|_U$ for $U$ an open set on which $\mathrm{dev}$ is a diffeomorphism onto its image, and $F$ satisfies $F_{xx}=-1$ and $F_{yy}=1$. This is easily achieved if the holonomy is trivial, or consists of an order two rotation, as follows.

\begin{lemma} \label{lm:trivial hol}
  Suppose that $\ZZ_i$ is a simple closed curve such that the holonomy $\rho(\eta)$ is either trivial, or a transformation $z\mapsto -z+c$. Then there exists a smooth function $f:\Sigma\to \R$, supported in $V_i$, satisfying \eqref{eq:cond signs} at every point of $\ZZ_i$.
\end{lemma}

\begin{proof}
If $\rho(\eta)(z)=-z+c$, we can post-compose $\mathrm{dev}$ with a translation, which has the effect of conjugating $\rho$, and we can thus assume that $\rho(\eta)(z)=-z$. Now, consider the function $F(x,y)=(-x^2+y^2)/2$, and define $\widetilde f:=F\circ \mathrm{dev}$. We claim that $\widetilde f$ descends to a function $f:V_i\to\R$. If $\rho(\eta)(z)=z$, then the claim trivially holds. If $\rho(\eta)(z)=-z$, we have $\widetilde f(\eta(p))=F\circ \mathrm{dev}(\eta(p))=F(-\mathrm{dev}(p))=\widetilde f(p)$ since $F(-z)=F(z)$. In both cases, $f$ satisfies \eqref{eq:cond signs} at every point of $\ZZ_i$ by Lemma \ref{lemma:euclidean hessian}. Then we conclude by multiplying $f$ by $\phi(d(\cdot,\ZZ_i))$, as for Lemma \ref{lm:point}. 
\end{proof}

\subsection{The interesting case}

We now consider the interesting case, namely when $\rho(\eta)$ is a translation $z\mapsto z+c$. This will require some additional technicalities. 
The statement we aim to prove is the following. 

\begin{lemma} \label{lm:curve}
  Suppose that $\ZZ_i$ is a simple closed curve such that the holonomy $\rho(\eta)$ is  a translation. Then there exists a smooth function $f:\Sigma\to \R$, supported in $V_i$, satisfying \eqref{eq:cond signs} at every point of $\ZZ_i$.
\end{lemma}

Let us explain the idea of the proof. One cannot simply proceed as in the proofs of Lemma \ref{lm:trivial hol} and Lemma \ref{lm:curve}, the reason being that the function $F(x,y)=(-x^2+y^2)/2$ (or more generally $F(x,y)=(-x^2+y^2)/2+\alpha(x,y)$, for $\alpha$ an affine function) is not invariant by translation. Indeed, $F(x+x_0,y+y_0)-F(x,y)$ is an affine function that depends on $x_0$ and $y_0$. The idea is then to start from a function of the form $F(x,y)=(-x^2+y^2)/2+\alpha(x,y)$ and modify it, still maintaining the condition \eqref{eq:cond signs} and concentrating the non-trivial modification   in a small open set, so as to achieve the right invariance by translation.

We will prove the following technical statement.

\begin{proposition}\label{prop:technical statement}
Let $\zeta(x)=(x,h(x))$ be a smooth curve in $\R^2$ whose image is not contained in any line, for $x\in (0,\delta)$. Fix $(x_0,y_0)$ in $\R^2$. Then there exists a smooth function $F:\R^2\to\R$ and an affine function $\alpha:\R^2\to\R$ such that, denoting $\Psi_\alpha(x,y)=(-x^2+y^2)/2+\alpha(x,y)$,
\begin{enumerate}
\item On $\{x\leq 0\}$, $F(x,y)=\Psi_\alpha(x,y)$;
\item On $\{x\geq \delta\}$, $F(x,y)=\Psi_\alpha(x-x_0,y-y_0)$;
\item The second derivatives of $F$ satisfy $F_{xx}=-1$ and $F_{yy}=1$ at every point of $\zeta$.
\end{enumerate}
\end{proposition}

Before proving Proposition \ref{prop:technical statement}, we explain how it will imply Lemma \ref{lm:curve}. 

\begin{proof}[Proof of Lemma \ref{lm:curve}]
We use the same notation as in Section \ref{subsec:easy case}. Assume the holonomy of $\ZZ_i$ is given, in real coordinates, by
\begin{equation}\label{eq:hol translation}
\rho(\eta)(x,y)=(x+x_0,y+y_0)~.
\end{equation}
Take coordinates $(t,s)\in \R/L\Z\times(-\bar s,\bar s)$ on $V_i$ such that the simple closed curve $\ZZ_i$ in $V_i$ is parameterized by $t\mapsto (t,0)$. Then $\widetilde V_i$ has induced coordinates $(t,s)\in\R\times (-\bar s,\bar s)$ (that we still denote with the same letters by a small abuse of notation) and the generator $\eta$ of $\pi_1(\widetilde V_i)\cong\Z$ acts by deck transformation as $\eta(t,s)=(t+L,s)$. 

Now, choose a developing map $\mathrm{dev}:\widetilde V_i\to\C$ and let $\zeta:\R\to\C$ be the curve $\zeta(t)=\mathrm{dev}(t,0)$ in the above coordinates. Take $a,b,\epsilon$ with $0<\epsilon<a<b<L-\epsilon$ and such that $\mathrm{dev}$ is a diffeomorphism onto its image when restricted to $[a,b]\times (-\bar s,\bar s)$, up to taking a smaller $\bar s$ if necessary. Since by Proposition \ref{prop:no geodesic} $\zeta$ is nowhere contained in a line, its tangent vector is generically not vertical, so up to restricting the interval $[a,b]$ we can assume that $\zeta|_{[a,b]}$ is a graph over the $x$-axis. Finally, post-composing $\mathrm{dev}$ with a translation we also assume  that $\zeta(a)=0$. 

We can apply Proposition \ref{prop:technical statement} and find a function $F:\R^2\to\R$ such that $F_{xx}=-1$ and $F_{yy}=1$ on every point of $\zeta$, with $F(x,y)=\Psi_\alpha(x,y)=(-x^2+y^2)/2+\alpha(x,y)$ on $\{x\leq 0\}$ for some affine function $\alpha$, and $F(x,y)=\Psi_\alpha(x-x_0,y-y_0)$ on $\{x\geq \delta\}$, where $\delta$ is such that $\zeta(b)\in \{x=\delta\}$. We define a function $\widetilde f$ as follows.
$$\widetilde f(t,s)=\begin{cases}
 \Psi_\alpha(\mathrm{dev}(t,s)) & \textrm{if }t\in (-\epsilon,a] \\
F(\mathrm{dev}(t,s)) & \textrm{if }t\in[a,b] \\
\Psi_\alpha(\mathrm{dev}(t,s)-(x_0,y_0))  & \textrm{if }t\in [b,L+\epsilon)
\end{cases}$$
This function is defined on the points of $\widetilde V_i$ such that $t\in(-\epsilon,L+\epsilon)$. It is well-defined and smooth by the properties of $F$ from Proposition \ref{prop:technical statement}. 

Moreover, it satisfies $\widetilde f\circ \eta=\widetilde f$ because for $t\in (-\epsilon,\epsilon)$, $\eta(t,s)=(t+L,s)$, hence
\begin{align*}
\widetilde f\circ\eta(t,s)&=\Psi_\alpha(\mathrm{dev}(\eta(t,s))-(x_0,y_0)) \\
&=\Psi_\alpha(\rho(\eta)\mathrm{dev}(t,s)-(x_0,y_0)) \\
&=\Psi_\alpha(\mathrm{dev}(t,s)) \\
&=\widetilde f(t,s)~,
\end{align*}
where in the second last line we have used \eqref{eq:hol translation}. 

Therefore $\widetilde f$ induces a smooth function on $\{(t,s)\in V_i\,|\,|s|<\bar s\}$ that, by Lemma \ref{lemma:euclidean hessian}, satisfies the property \eqref{eq:cond signs} on $\ZZ_i$.
Multiplying by $\phi(d(\cdot,\ZZ_i))$, for $\phi$ a bump function as in Definition \ref{defi:bump} and $r$ such that $N_r(\ZZ_i)\subset \{|s|<\bar s\}$, we obtain the desired function $f$.
\end{proof}

We conclude the section by proving Proposition \ref{prop:technical statement}.

\begin{proof}[Proof of Proposition \ref{prop:technical statement}]
We will construct a function $G:\R^2\to\R$ such that:
\begin{enumerate}
\item[\emph{(1')}] On $\{x\leq 0\}$, $G(x,y)=\Psi_0(x,y)=(-x^2+y^2)/2$;
\item[\emph{(2')}] On $\{x\geq \delta\}$, $G(x,y)=\Psi_0(x-x_0,y-y_0)+C$ for some $C\in\R$;
\item[\emph{(3')}] The second derivatives of $G$ satisfy $G_{xx}=-1$ and $G_{yy}=1$ on every point of $\zeta$.
\end{enumerate}

Then the desired $F$ is obtained as $F=G+\alpha$, where $\alpha$ is an affine function such that $\alpha(x+x_0,y+y_0)=\alpha(x,y)-C$. 

We shall define the function $G$ through its Hessian. That is, consider the field of 2-by-2 matrices $\mathcal H:\R^2\to\mathcal M_{2,2}(\R)$ given by:
$$\mathcal H(x,y)=\begin{pmatrix} -1+(y-h(x))\xi'(x) & \xi(x) \\ \xi(x) & 1 \end{pmatrix}$$
for a smooth function $\xi:\R\to\R$, to be determined, such that $\xi$ is supported on $(0,\delta)$. One easily checks that $\mathcal H$ is the Hessian of a function defined on $\R^2$, because it is symmetric and satisfies the Codazzi condition. That is, the 1-forms $\omega_1:=(-1+(y-h(x))\xi'(x))dx+\xi(x)dy$ and $\omega_2:=\xi(x)dx+dy$ obtained from the columns of $\mathcal H$ are closed. Hence there exist two functions $G_1$ and $G_2$ such that $dG_1=\omega_1$ and $dG_2=\omega_2$. Moreover $G_1dx+G_2dy$ is in turn closed by symmetry of $\mathcal H$, and thus equals $dG$ for some function $G:\R^2\to\R$. It follows that $\mathcal H=D^2G$.

Observe that $\mathcal H=\mathrm{diag}(-1,1)$ on $\{x\leq 0\}$, on $\{x\geq \delta\}$. Moreover, on the image of $\zeta$, the term $(y-h(x))$ vanishes, so there $\mathcal H=\mathrm{diag}(-1,1)$ again. Therefore the condition \emph{(3')} is satisfied. Up to adding an affine function, we can further assume that $G=(-x^2+y^2)/2$ on $\{x\leq 0\}$, and thus  \emph{(1')} is satisfied. It thus remains to check \emph{(2')}, that will require suitably choosing $\xi$.

Since 
$\Psi_0(x-x_0,y-y_0)=(-x^2+y^2)/2+x_0x-y_0y-x_0^2/2+y_0^2/2$, in order to achieve \emph{(2')} it suffices to guarantee that $G_1(x,y)=-x+x_0$ and $G_2(x,y)=y-y_0$ on $\{x\geq \delta\}$. Integrating along $\zeta(x)=(x,h(x))$, we only need to impose:
\begin{equation}\label{eq:integrate}
\int_0^\delta \xi(x)h'(x)dx=x_0\qquad \textrm{and}\qquad \int_0^\delta \xi(x)dx=-y_0~.
\end{equation}
To conclude the proof, we claim that there exists a function $\xi$ for which \eqref{eq:integrate} is satisfied. For this purpose, we show that the following linear map $\mathcal I:C^\infty_0((0,\delta))\to\R^2$ is surjective:
$$\xi\mapsto\left(\int_0^\delta \xi(x)h'(x)dx, \int_0^\delta \xi(x)dx\right)~.$$
Clearly the rank of $\mathcal I$ is at least one, because the integral of $\xi$ does not vanish for some function $\xi$. So, suppose by contradiction that the image of $\mathcal I$ is the span of $(\lambda,1)$ for some $\lambda\in\R$. This means that 
$$\int_0^\delta \xi(x)h'(x)dx=\lambda \int_0^\delta \xi(x)dx$$
or equivalently 
$$\int_0^\delta \xi(x)(h'(x)-\lambda)dx=0$$
for every $\xi$. By the fundamental lemma of the calculus of variations, $h'(x)\equiv \lambda$, which implies that $\zeta$ is contained in a line, contradicting the hypothesis.
\end{proof}

\section{Conclusion of the proofs}

Let us now wrap up the proof of Theorem \ref{thm:main}.

\begin{proof}[Proof of Theorem \ref{thm:main}]
Let $\Sigma$ be an embedded minimal surface in a hyperbolic three-manifold, with principal curvatures in $[-1,1]$. Recall that $\ZZ$, which has been defined in Definition \ref{defi ZZ}, consists of finitely many isolated points and simple closed curves. Combining Lemmas \ref{lm:point}, \ref{lm:trivial hol} and \ref{lm:curve}, we can find a smooth function $f:\Sigma\to\R$ satisfying the condition \eqref{eq:cond signs}, which is obtained as the sum of smooth functions supported in the pairwise disjoint $r$-neighbourhoods of each connected component of $\ZZ$.  

Now, recall that by Lemma \ref{lemma:embedded} the surface $\Sigma_{t\!f}$ obtained by normal evolution along $tf$ is embedded for $t\in(-\epsilon,\epsilon)$. By Lemma \ref{lemma:variation principal curv}, at every point of $\ZZ$ the derivative of the positive principal curvature $\lambda_{t\!f}^+$ of $\Sigma_{t\!f}$ is $-1$, while the derivative of the negative principal curvature $\lambda^-_{t\!f}$ is $+1$. Hence, by continuity, for every $p\in\ZZ$ there exists $\epsilon(p)$ and a neighbourhood $\mathcal U_p$ such that $\iota_{t\!f}(\mathcal U_p)$ has principal curvatures in $(-1,1)$ for $t\in(-\epsilon(p),\epsilon(p))$. By compactness of $\ZZ$, we can extract a finite subcover from the covering $\{\mathcal U_p\}_{p\in\ZZ}$, and therefore one can find a neighbourhood $\mathcal U$ of $\ZZ$ such that $\iota_{t\!f}(\mathcal U)$ has principal curvatures in $(-1,1)$ for $t\in(-\epsilon,\epsilon)$, up to taking a small enough $\epsilon$.

It remains to ensure that the principal curvatures of $\Sigma_{t\!f}=\iota_{t\!f}(\Sigma)$ are in  $(-1,1)$ everywhere. Using compactness, we can fix $\tau_0\in (0,1)$ such that 
$$\{p\in\Sigma\,|\,\lambda^+(p)>1-\tau_0\textrm{ and }\lambda^-(p)<-1+\tau_0\}\subset \mathcal U~.$$
Then by continuity and compactness again, the principal curvatures of $\iota_{t\!f}(\Sigma\setminus \mathcal U)$ are in $(-1+\tau_0/2,1-\tau_0/2)$ for $t\in(-\epsilon,\epsilon)$, taking a smaller  $\epsilon$ if necessary. So in conclusion, the principal curvatures of $\Sigma_{t\!f}$ are in $(-1,1)$ everywhere, for $t$ sufficiently small.
 \end{proof}

We conclude by a quick proof of the partial converse result, namely Theorem \ref{thm:converse}.

\begin{proof}[Proof of Theorem \ref{thm:converse}]
Let $\Sigma$ be a closed embedded surface in $M$ with principal curvatures in $(-\varepsilon,\varepsilon)$. Lift $\Sigma$ to an embedded surface $\widetilde\Sigma$ in $\widetilde M\cong\HH^3$. From the results of \cite[Section 5]{epstein2}, $\partial_\infty\widetilde\Sigma$ is a $K$-quasicircle for 
$$K=\frac{1+\varepsilon}{1-\varepsilon}~.$$
From the results of \cite{seppi} (see also \cite[Corollary 1.7]{huang-lowe-seppi}), there exists a universal constant $K_0>1$ such that, if $\widehat \Sigma$ is a properly embedded minimal surface in $\HH^3$ with asymptotic boundary a $K$-quasicircle with $K\leq K_0$, then $\widehat \Sigma$ has principal curvatures in $(-1,1)$. Choosing $\varepsilon$ such that $(1+\varepsilon)/(1-\varepsilon)\leq K_0$ and applying the result to $\widehat\Sigma=\widetilde\Sigma$ concludes the proof. 
 \end{proof}

\bibliographystyle{alpha}
\bibliography{bibnearly}

\end{document}